\documentclass[oneside,english]{amsart}
\usepackage{mathpazo}
\usepackage[T1]{fontenc}
\usepackage[latin9]{inputenc}
\usepackage{amsthm}
\usepackage{amssymb}

\makeatletter
\numberwithin{equation}{section}
\numberwithin{figure}{section}
\usepackage{enumitem}		
\theoremstyle{plain}
\newtheorem{thm}{\protect\theoremname}
  \theoremstyle{plain}
  \newtheorem{cor}[thm]{\protect\corollaryname}
  \theoremstyle{plain}
  \newtheorem{lem}[thm]{\protect\lemmaname}
  \theoremstyle{definition}
  \newtheorem{defn}[thm]{\protect\definitionname}
  \theoremstyle{remark}
  \newtheorem*{claim*}{\protect\claimname}
  \theoremstyle{definition}
  \newtheorem{problem}[thm]{\protect\problemname}
  \theoremstyle{plain}
  \newtheorem{conjecture}[thm]{\protect\conjecturename}

\usepackage{enumitem}
\newtheorem*{myConvention}{Convention}

\makeatother

\usepackage{babel}
  \providecommand{\claimname}{Claim}
  \providecommand{\conjecturename}{Conjecture}
  \providecommand{\corollaryname}{Corollary}
  \providecommand{\definitionname}{Definition}
  \providecommand{\lemmaname}{Lemma}
  \providecommand{\problemname}{Problem}
\providecommand{\theoremname}{Theorem}

\begin{document}

\author{Micha\l{} Adamaszek}

\address{Department of Mathematical Sciences, University of Copenhagen, Universitetsparken
5, 2100 Copenhagen, Denmark}

\thanks{MA is supported by VILLUM FONDEN through the network for Experimental
Mathematics in Number Theory, Operator Algebras, and Topology.}

\email{aszek@mimuw.edu.pl}

\author{Jan Hladký}

\address{Institute of Mathematics, Czech Academy of Science. Žitná 25, 110
00, Praha, Czech Republic. The Institute of Mathematics of the Czech
Academy of Sciences is supported by RVO:67985840.}

\thanks{The research leading to these results has received funding from the
People Programme (Marie Curie Actions) of the European Union's Seventh
Framework Programme (FP7/2007-2013) under REA grant agreement umber
628974}

\email{honzahladky@gmail.com}



\title{Upper bound theorem for odd-dimensional flag triangulations of manifolds}
\begin{abstract}
We prove that among all flag triangulations of manifolds of odd dimension
$2r-1$ with sufficiently many vertices the unique maximizer of the
entries of the $f$-, $h$-, $g$- and $\gamma$-vector is the balanced
join of $r$ cycles. Our proof uses methods from extremal graph theory.
\end{abstract}

\maketitle

\section{Introduction}

The classification of face numbers ($f$-vectors) of various classes
of simplicial complexes, especially triangulations of spheres, balls
and manifolds, is a classical topic in enumerative combinatorics.
The Charney--Davis conjecture~\cite{CharneyDavis} and its generalization
by Gal~\cite{SRGal} sparkled the interest in similar questions for
the class of flag triangulations. In this paper we prove a general
upper bound theorem for flag triangulations of odd-dimensional manifolds.

A simplicial complex $K$ is \emph{flag} if every set of vertices
pairwise adjacent in the $1$-skeleton of $K$ spans a face of $K$
or, equivalently, if $K$ is the clique complex of its $1$-skeleton.
Flag complexes appear prominently in Gromov's approach to non-positive
curvature (see~\cite{Gromov} and~\cite{metricCharney} for an exposition).
In this context Charney and Davis proposed their famous conjecture~\cite{CharneyDavis}
that a certain linear combination of the face numbers of any odd-dimensional
flag triangulation of a homology sphere is non-negative. Subsequently,
Gal~\cite{SRGal} introduced a modification of the $f$-vector called
the $\gamma$-vector, which seems well-suited to the study of flag
triangulations of homology spheres, and is conjecturally non-negative.
Since then a number of conjectures have been made about the structure
of $\gamma$-vectors of flag triangulations of spheres, with many
of them verified in special cases~\cite{AdaHla:DenseFlagTriangulations,aisbett2014frankl,SRGal,karu2006cd,murai2012cd,nevo2011gamma,nevo2011vector}.
Note that a flag complex is completely determined by its $1$-skeleton,
and thus its face vector is the clique vector of the underlying graph.
Paradoxically, this only adds to the complexity of the problem. For
example, face vectors of arbitrary simplicial complexes are characterized
by the Kruskal--Katona theorem, while the clique vectors of general
graphs are not so well understood~\cite{andy:clique}.

Our contribution is an upper bound theorem for odd-dimensional flag
triangulations of homology manifolds, a class which includes flag
triangulations of simplicial manifolds and flag triangulations of
homology spheres. We exhibit a unique maximizer of any reasonable
linear combination of face numbers. For any $r\geq1$ and $n\geq4r$
let $\mathbf{\mathbf{J}}_{r}(n)$ be the $n$-vertex flag complex
obtained as a join of $r$ copies of the circle $S^{1}$, each one
a cycle with $\lfloor\frac{n}{r}\rfloor$ or $\lceil\frac{n}{r}\rceil$
vertices. This complex is a flag triangulation of a simplicial $(2r-1)$-sphere.
To phrase our main theorem we say that a real-valued function $F$
defined on simplicial complexes is a \emph{face function in dimension
$\ell$} if it can be written as $F(K)=c_{\ell}f_{\ell}(K)+c_{\ell-1}f_{\ell-1}(K)+\ldots+c_{0}f_{0}(K)+c_{-1}$
with $c_{i}\in\mathbb{R}$ and $c_{\ell}>0$, where $f_{i}(K)$ is
the number of $i$-dimensional faces of $K$.
\begin{thm}[Main theorem]
\label{thm:Main} For every even number $d\geq4$ and every face
function $F$ in dimension $\ell$, where $1\leq\ell\leq\frac{d}{2}-1$,
there exists a constant $n_{0}$ for which the following holds. If
$M$ is a flag triangulation of a homology manifold of dimension $d-1$
with $n\geq n_{0}$ vertices then 
\[
F(M)\leq F(\mathbf{J}_{\frac{d}{2}}(n))
\]
 and equality holds if and only if $M$ is isomorphic to $\mathbf{\mathbf{J}}_{\frac{d}{2}}(n)$.
\end{thm}

In this context the standard statistics based on face numbers are
the $f$-vector $(f_{-1},f_{0},\ldots,f_{d-1})$, the $h$-vector
$(h_{0,}h_{1,}\ldots,h_{d})$, the $g$-vector $(g_{0},g_{1},\ldots,g_{\frac{d}{2}})$
and the $\gamma$-vector $(\gamma_{0},\gamma_{1},\ldots,\gamma_{\frac{d}{2}})$.
Theorem~\ref{thm:Main} specializes to an upper bound statement for
all of those simultaneously.
\begin{cor}
\label{cor:fhgamma}For every even number $d\geq4$ there is a constant
$N_{0}$ for which the following holds. If $M$ is a flag triangulation
of homology manifold of dimension $d-1$ with $n\geq N_{0}$ vertices
then
\begin{eqnarray*}
f_{i}(M)\leq f_{i}(\mathbf{J}_{\frac{d}{2}}(n)) & \mathrm{for} & 1\leq i\leq d-1,\\
h_{i}(M)\leq h_{i}(\mathbf{J}_{\frac{d}{2}}(n)) & \mathrm{for} & 2\leq i\leq d-2,\\
g_{i}(M)\leq g_{i}(\mathbf{J}_{\frac{d}{2}}(n)) & \mathrm{for} & 2\leq i\leq\tfrac{d}{2},\\
\gamma_{i}(M)\leq\gamma_{i}(\mathbf{J}_{\frac{d}{2}}(n)) & \mathrm{for} & 2\leq i\leq\tfrac{d}{2}.
\end{eqnarray*}
Moreover, equality in any of these inequalities implies that $M$
is isomorphic to $\mathbf{J}_{\frac{d}{2}}(n)$.
\end{cor}
For all other values of the index $i$, as well as for face functions
in dimension $0$ or $-1$ in Theorem~\ref{thm:Main}, the corresponding
inequalities are trivially satisfied with equality for all $M$. We
note that all entries of the $f$-, $h$- and $g$-vectors are linear
combinations of the entries of the $\gamma$-vector with \emph{non-negative
}coefficients (see next section), therefore the inequalities for the
$\gamma$-vector imply all the remaining ones. 

The only previously known case of Corollary~\ref{cor:fhgamma} was
$d=4$ (for any $n$) due to Gal~\cite{SRGal}, with the uniqueness
part (for large $n$) following from~\cite{AdaHla:DenseFlagTriangulations}.
In this case all inequalities follow from $f_{1}(M)\leq f_{1}(\mathbf{J}_{2}(n))$.
Lutz and Nevo~\cite[Conjecture 6.3]{nevolutz} conjectured the upper
bound for the $\gamma$-vector for flag triangulations of odd-dimensional
spheres. Our result confirms this conjecture for large $n$. It also
provides supporting evidence for a question of Nevo and Petersen~\cite[Problem 6.4]{nevo2011gamma}
(see Section~\ref{sec:Conjectures} for details). We also previously
conjectured the upper bound on $f_{1}$ for arbitrary even $d$ in~\cite{AdaHla:DenseFlagTriangulations}. 

For arbitrary (not necessarily flag) triangulations of odd-dimensional
homology manifolds tight upper bounds for $f_{i}$ were obtained by
Novik~\cite[Theorem 1.4]{NovikUBC}. In this case the maximum is
attained by the boundary of the $d$-dimensional cyclic polytope with
$n$ vertices (the maximizer is not unique). For the subclass of simplicial
spheres this had been known before by the celebrated upper bound theorem
of Stanley~\cite{stanley1975upper}. In the flag case our result
is new even for flag triangulations of simplicial spheres.

\section{Preliminaries}

We recommend the reader~\cite{stanley2004combinatorics} and~\cite{NovikUBC}
as references for face numbers of triangulations of manifolds and
spheres.

An \emph{abstract simplicial complex }$K$ with vertex set $V$ is
a collection $K\subseteq2^{V}$ such that $\sigma\in K$ and $\tau\subseteq\sigma$
imply $\tau\in K$. The elements of $K$ are called \emph{faces. }The
\emph{dimension} of $\sigma$ is $|\sigma|-1$ and the dimension of
$K$ is the maximal dimension of any of its faces. The \emph{link
}of a face $\sigma$ is the subcomplex $\mathrm{lk}_{K}(\sigma)=\{\tau\in K\colon\tau\cap\sigma=\emptyset,\ \tau\cup\sigma\in K\}$. 

A simplicial complex $K$ is a \emph{simplicial manifold }(resp. \emph{simplicial
sphere) }of dimension $q$ if the geometric realization $|K|$ is
homeomorphic to a connected, compact topological $q$-manifold without
boundary (resp. to the sphere $S^{q}$). Most known results involving
face numbers of simplicial manifolds hold for more general objects,
which we now introduce. A simplicial complex $K$ is a \emph{homology
manifold }if for any point $p\in|K|$ and any $i\neq\dim K$, $H_{i}(|K|,|K|-p;\mathbb{Z})=0$
and $H_{\dim K}(|K|,|K|-p;\mathbb{Z})=\mathbb{Z}$ (throughout this
paper, all homology groups are taken with integer coefficients). This
is equivalent to saying that for every nonempty face $\sigma\in K$
the link $\mathrm{\mathrm{lk}}_{K}(\sigma)$ has the homology of the
sphere $S^{q-|\sigma|}$ (equivalence follows from the excision axiom,
see \cite[Lemma 3.3]{munkres1984}). A \emph{homology sphere }is a
homology manifold $K$ such that $K$ itself has the homology of a
sphere. It is easy to see that if $K$ is a homology $q$-manifold
then for every nonempty face $\sigma\in K$ the link $\mathrm{\mathrm{lk}}_{K}(\sigma)$
is a homology $(q-|\sigma|)$-sphere. Clearly every simplicial manifold
(resp. simplicial sphere) is a homology manifold (resp. homology sphere).

A complex $K$ of dimension $q$ is called \emph{Eulerian }if for
every face $\sigma\in K$ (including the empty one) the link $\mathrm{lk}_{K}(\sigma)$
has the same Euler characteristic as the sphere $S^{q-|\sigma|}$.
Every homology manifold satisfies Poincaré duality; as a consequence
the Euler characteristic of an odd-dimensional homology manifold $M$
equals $0$ and so $M$ is Eulerian.

For a $(d-1)$-dimensional complex $K$ with $n$ vertices let $f_{i}(K)$
be the number of $i$-dimensional faces. The vector $(f_{-1},f_{0},\ldots,f_{d-1})$
is called the \emph{$f$-vector }of $K$ (note that $f_{-1}(K)=1$
and $f_{0}(K)=n$). The $h$-vector $(h_{0},h_{1},\ldots,h_{d})$
of $K$ is a convenient modification of the $f$-vector defined by
the identity 
\begin{equation}
\sum_{i=0}^{d}f_{i-1}x^{i}(1-x)^{d-i}=\sum_{i=0}^{d}h_{i}x^{i}.\label{eq:hvector}
\end{equation}
Note $h_{0}(K)=1$ and $h_{1}(K)=n-d$. An Eulerian simplicial complex
satisfies the Dehn--Sommerville equations $h_{i}(K)=h_{d-i}(K)$ for
$0\leq i\leq d$. In that case one can define the $\gamma$-vector
$(\gamma_{0},\ldots,\gamma_{\lfloor\frac{d}{2}\rfloor})$ of $K$
by the identity 
\begin{equation}
\sum_{i=0}^{d}h_{i}x^{i}=\sum_{i=0}^{\lfloor d/2\rfloor}\gamma_{i}x^{i}(x+1)^{d-2i}.\label{eq:gammavector}
\end{equation}
Here $\gamma_{0}(K)=1$ and $\gamma_{1}(K)=n-2d$. The $\gamma$-vector
was first introduced by Gal~\cite{SRGal} for flag triangulation
of homology spheres, for which it is conjectured to be non-negative.
This conjecture generalizes the Charney--Davis conjecture, which in
this language asserts that $\gamma_{\frac{d}{2}}(K)$ is non-negative
for a $(d-1)$-dimensional flag triangulation of a homology sphere
$K$ with $d$ even. Another classical invariant, studied mostly for
simplicial spheres and balls, is the $g$-vector $(g_{0},g_{1},\ldots,g_{\lfloor\frac{d}{2}\rfloor})$
given by $g_{0}=1$ and $g_{i}=h_{i}-h_{i-1}$ for $1\leq i\leq\lfloor\frac{d}{2}\rfloor$.

Suppose now that $d$ is even and let $M$ be a homology $(d-1)$-manifold.
For any $i$ the function $f_{i}(M)$ is clearly a face function in
dimension $i$. For any $0\leq i\leq d$ we have $h_{i}=\sum_{j=0}^{i}(-1)^{j-i}{d-j \choose i-j}f_{j-1}$,
so $h_{i}(M)$ is a face function in dimension $i-1$. By the Dehn--Sommerville
equations if $\frac{d}{2}\leq i\leq d$ then $h_{i}(M)=h_{d-i}(M)$
can be expressed as a face function in dimension $d-i-1$. For any
$0\leq i\leq d-1$ we have $f_{i}=\sum_{j=0}^{i+1}{d-j \choose i+1-j}h_{j}$.
If $\frac{d}{2}\leq i\leq d-1$ the Dehn--Sommerville equations imply
that $f_{i}$ is a linear combination of $(h_{\frac{d}{2}},\ldots,h_{0})$
with leading term ${d/2 \choose i+1-d/2}h_{\frac{d}{2}}$, so by the
previous observations $f_{i}(M)$ is equal to a face function in dimension
$\frac{d}{2}-1$. Finally both $\gamma_{i}$ and $g_{i}$ are linear
combinations of $(h_{i},\ldots,h_{0})$ with leading term $h_{i}$,
hence $\gamma_{i}(M)$ and $g_{i}(M)$ are face functions in dimension
$i-1$. Using these observations the proof of Corollary~\ref{cor:fhgamma}
from Theorem~\ref{thm:Main} is immediate.

\medskip{}

Let us now move towards flag complexes. If $G=(V,E)$ is a finite,
simple, undirected graph then the \emph{clique number} $\omega=\omega(G)$
of $G$ is the cardinality of the largest clique (complete subgraph)
in $G$ and the \emph{clique vector }of $G$ is the sequence $(e_{0}(G),e_{1}(G),\ldots,e_{\omega}(G))$,
where $e_{i}(G)$ is the number of cliques of cardinality $i$ (in
particular $e_{0}(G)=1$, $e_{1}(G)=|V|$ and $e_{2}(G)=|E|$). The
\emph{clique complex }of $G$, denoted $X(G)$, is the simplicial
complex with vertex set $V$ whose faces are all cliques in $G$.
We have $\dim X(G)=\omega(G)-1$ and $f_{i}(X(G))=e_{i+1}(G)$. Note
that the $1$-skeleton of $X(G)$ is $G$. A simplicial complex is
\emph{flag }if it is the clique complex of a graph. A \emph{flag triangulation
of a homology manifold }(resp. \emph{flag triangulation of a homology
sphere) }is a flag complex which is a triangulation of a homology
manifold (resp. a homology sphere).

\medskip{}

\begin{myConvention}By abuse of language we will say that \emph{$G$
triangulates a homology manifold }(resp.\emph{ sphere}) if $X(G)$
is a flag triangulation of a homology manifold (resp. sphere).\end{myConvention}\medskip{}

Fix $n,r\in\mathbb{N}$. We write $T_{r}(n)$ for the $r$-partite
Turán graph of order $n$, that is a graph with $n$ vertices partitioned
into sets $V_{1},V_{2},\ldots,V_{r}$, each of size either $\lfloor\frac{n}{r}\rfloor$
or $\lceil\frac{n}{r}\rceil$, with no edge inside any $V_{i}$ and
with a complete bipartite graph between every two $V_{i}$ and $V_{j}$,
$i\neq j$. Further, for $n\geq4r$ we define $J_{r}(n)$ to be the
graph obtained from $T_{r}(n)$ by declaring that each of the parts
$V_{i}$ induces a cycle of length $|V_{i}|$. The condition $n\geq4r$
guarantees that each part is a cycle of length at least $4$, hence
a flag triangulation of $S^{1}$. Of course we have $\mathbf{J}_{r}(n)=X(J_{r}(n))$
and this complex is a flag simplicial $(2r-1)$-sphere.

We say that a real-valued function $F$ defined on graphs is a \emph{clique
function of order $k$}, if $F$ can be written as $F(G)=c_{k}e_{k}(G)+c_{k-1}e_{k-1}(G)+\cdots+c_{1}e_{1}(G)+c_{0}$
where $c_{i}\in\mathbb{R}$ and $c_{k}>0$. Theorem~\ref{thm:Main}
can be equivalently rephrased as follows.
\begin{thm}[Main Theorem, Graph formulation]
\label{thm:MainGraphTheory} For every $r\geq2$ and every clique
function $F$ of order $k$, where $2\leq k\leq r$, there exists
a constant $n_{0}$ for which the following holds. If $G$ is a graph
with $n\geq n_{0}$ vertices which triangulates a $(2r-1)$-dimensional
homology manifold then 
\[
F(G)\leq F(J_{r}(n))
\]
 and equality holds if and only if $G$ is isomorphic to $J_{r}(n)$.
\end{thm}
Let us first fix some additional notation. The \emph{neighborhood}
of a vertex $v$ in a graph $G$ is the set $N_{G}(v)=\{w\colon\left\{ v,w\right\} \in E(G)\}$
and for a clique $\sigma$ in $G$ we define the \emph{link of $\sigma$
in $G$ }as the induced subgraph $\mathrm{lk}_{G}(\sigma)=G[\bigcap_{v\in\sigma}N_{G}(v)]$.
This notation is designed so that $X(\mathrm{lk}_{G}(\sigma))=\mathrm{lk}_{X(G)}(\sigma)$.
For a vertex $v\in V(G)$ and a subset $W\subset V(G)$ we write $\deg_{G}(v)=|N_{G}(v)|$
and $\deg_{G}(v,W)=|N_{G}(v)\cap W|$. The subscript $G$ will be
omitted if there is no risk of confusion.

\subsection{Properties of flag triangulations of homology manifolds}

Below, we record two basic properties of flag triangulations of homology
manifolds that we need for our proof of the Main Theorem.
\begin{lem}
\label{lem:middleDS}For every $r\geq1$ there is a constant $C_{r}$
such that every $n$-vertex graph $G$ triangulating a $(2r-1)$-dimensional
homology manifold satisfies $e_{r+1}(G)\leq C_{r}n^{r}$.\end{lem}
\begin{proof}
The Dehn--Sommerville relation $h_{r+1}(X(G))=h_{r-1}(X(G))$ expressed
in terms of face numbers implies that $f_{r}(X(G))$ is a linear combination
of entries of the vector $(f_{r-1}(X(G)),\ldots,f_{-1}(X(G)))=(e_{r}(G),\ldots,e_{0}(G))$
with coefficients depending only on $r$. Since $e_{i}(G)\leq{n \choose i}\leq n^{r}$
for $0\leq i\leq r$, we get $e_{r+1}(G)=f_{r}(X(G))\leq C_{r}n^{r}$
for a suitable $C_{r}.$
\end{proof}

Let $K_{3}^{r}:=T_{r}(3r)$ denote the complete $r$-partite graph
with all parts of size $3$. A graph $G$ is $H$-free if it does
not contain $H$ as a subgraph. The crucial geometric ingredient of
our arguments is provided by the next lemma.
\begin{lem}
\label{lem:vankampen}Fix $r\geq1$. If $G$ triangulates a homology
sphere of dimension $2(r-1)$ then $G$ is $K_{3}^{r}$-free.\end{lem}
\begin{proof}
By a result of Galewski and Stern~\cite[Corollary 1.9]{galewski1980classification},
if $X(G)$ is a homology $2(r-1)$-sphere then the double suspension
$\Sigma^{2}X(G)$ is homeomorphic to $S^{2r}$. Now if $G$ contained
$K_{3}^{r}$ then $\Sigma^{2}X(G)$ would contain an embedded $X(K_{3}^{r+1})$,
formed by the original $K_{3}^{r}$ and any three of the four suspending
vertices. That contradicts the theorem of van Kampen and Flores \cite{vanKampen1932,flores1933}
(see also~\cite[Section 2.4]{wagner:minors}) that $X(K_{3}^{r+1})$
is not embeddable in $S^{2r}$. 
\end{proof}
In our arguments we are going to apply Lemma~\ref{lem:vankampen}
to links of faces in a homology manifold. For example, we get that
if $G$ triangulates a homology $(2r-1)$-manifold then for every
vertex $v$ the link $\mathrm{lk}_{G}(v)$ is $K_{3}^{r}$-free.

\subsection{Extremal graph theory}

The remaining tools for our proof come entirely from extremal graph
theory. An approach to face enumeration via extremal graph theory
was pioneered in~\cite{AdaHla:DenseFlagTriangulations} where we
classified all flag triangulations of homology $3$-manifolds $M$
with a sufficiently large number $n$ of vertices which are almost
extremal for $f_{1}$ or $\gamma_{2}$. Thus, the main technical contribution
of our current work is in connecting further tools from extremal graph
theory (namely Zykov's inequalities (Theorem~\ref{thm:Zykov}) and
the Removal lemma (Theorem~\ref{thm:Removal})) to the area of face
enumeration.

\medskip{}

The following definition introduces a distance --- sometimes called
the \emph{edit distance} --- on the set of $n$-vertex graphs.
\begin{defn}
\label{def:distance}We say that two graphs with the same number of
$n$ vertices are \emph{$\epsilon$-close} if there exists an identification
of their vertex sets, so that then one graph can be obtained from
the other by editing (i.e., adding or deleting) less than $\epsilon n^{2}$
edges.
\end{defn}
The celebrated Stability Theorem of Erd\H{o}s and Simonovits~\cite{Erd:Stability,Sim:Stability}
below says that a $K_{r+1}$-free graph whose number of edges is close
to the Turán bound must actually be close to the Turán graph in the
edit distance.
\begin{thm}
\label{thm:ErdSimStability}Suppose that $r\geq2$ and $\epsilon>0$
are given. Then there exists $\delta>0$ such that whenever $H$ is
an $n$-vertex, $K_{r+1}$-free graph with $e_{2}(H)>(1-\delta)e_{2}(T_{r}(n))$
then $H$ is $\epsilon$-close to $T_{r}(n)$.
\end{thm}
We will also make use of the following result.
\begin{thm}
\label{thm:Zykov}Let $r\geq1$ and suppose than $H$ is an $n$-vertex,
$K_{r+1}$-free graph. Then we have
\[
1=\frac{e_{1}(H)}{e_{1}(T_{r}(n))}\ge\frac{e_{2}(H)}{e_{2}(T_{r}(n))}\ge\cdots\ge\frac{e_{r}(H)}{e_{r}(T_{r}(n))}\;.
\]

\end{thm}
Theorem~\ref{thm:Zykov} generalizes the result of Zykov~\cite{Zykov}
that $e_{k}(H)\leq e_{k}(T_{r}(n))$ for all $1\le k\le r$, which
in turn generalizes Turán's Theorem stating $e_{2}(H)\leq e_{2}(T_{r}(n))$.
A nice proof of Theorem~\ref{thm:Zykov} using symmetrization\emph{
}can be found in~\cite[Theorem 3.1]{goodarzi}.

\medskip{}

Let us now motivate the Removal lemma. A graph of order $n$ can contain
at most ${n \choose r+1}=\Theta(n^{r+1})$ copies of $K_{r+1}$. If
the graph is not complete then of course it contains fewer copies.
However, we think of the graph $H$ as ``essentially $K_{r+1}$-free''
if $e_{r+1}(H)=o(n^{r+1})$. It is then tempting to say that by removing
a few edges we can delete all the copies of $K_{r+1}$. This is true,
yet far from trivial, and a subject of the famous Removal lemma, a
form of which first appeared in~\cite{RusSze:RemovalLemma}, and
which was later formulated in its full strength in~\cite{ErdFraRod:RemovalLemma}.
\begin{thm}
\label{thm:Removal}Suppose that $r\geq1$ and $\alpha>0$ are given.
Then there exists $\beta>0$ such that whenever $H$ is an $n$-vertex
graph with $e_{r+1}(H)\leq\beta n^{r+1}$ then by deleting a suitable
set of less than $\alpha n^{2}$ edges $H$ can be made $K_{r+1}$-free.
\end{thm}

\subsection{Outline of the proof of Theorem~\ref{thm:MainGraphTheory}}

Suppose $G$ triangulates a homology $(2r-1)$-manifold and the number
of vertices $n$ is large. First, note that if $k\leq r$ then $e_{k}(T_{r}(n))\approx{r \choose k}\left(\frac{n}{r}\right)^{k}$
and $e_{k}(J_{r}(n))=e_{k}(T_{r}(n))+O(n^{k-1})={r \choose k}\left(\frac{n}{r}\right)^{k}+O(n^{k-1})$.
Now if $G$ is such that $e_{k}(G)\leq(1-\alpha)e_{k}(T_{r}(n))$,
for a fixed (but arbitrarily small) $\alpha>0$ then the inequality
$F(G)\leq F(J_{r}(n))$ follows just by comparing the terms of order
$n^{k}$ in $F$. 

Thus, we only need to deal with the case $e_{k}(G)>(1-\alpha)e_{k}(T_{r}(n))$.
Lemma~\ref{lem:middleDS} tells us that $G$ has only $O(n^{r})=o(n^{r+1})$
many $K_{r+1}$'s. Thus, we can apply Theorem~\ref{thm:Removal}
and obtain a $K_{r+1}$-free subgraph $G'\subset G$ with $e_{2}(G')=e_{2}(G)-o(n^{2})$.
Since deleting one edge from $G$ can decrease the number of copies
of $K_{k}$ by at most $n^{k-2}$, we have that $e_{k}(G')\ge e_{k}(G)-o(n^{k})>(1-2\alpha)e_{k}(T_{r}(n))$.
From Theorem~\ref{thm:Zykov} we can deduce that $e_{2}(G')\approx e_{2}(T_{r}(n))$.
At this point Theorem~\ref{thm:ErdSimStability} shows that $G'$
must be similar to $T_{r}(n)$. Since $G$ and $G'$ are similar,
we also get that $G$ must be similar to $T_{r}(n)$. Additional geometric
properties of $X(G)$ allow us to conclude from there that $F(G)$
is maximized by $F(J_{r}(n))$; this part is given in Lemmas~\ref{lem:radicalmanifold}-\ref{lem:optimizeExtremal}
below.

\subsection{Organisation of the paper}

As said earlier, the difficult cases in Theorem~\ref{thm:MainGraphTheory}
are those when $G$ is $\epsilon$-close to $T_{r}(n)$ for some $\epsilon>0$.
We will analyze their structure more closely in the next section.
In Section~\ref{sec:Proof} we give a proof of Theorem~\ref{thm:MainGraphTheory}.
We then conclude with open problems stemming from this work in Section~\ref{sec:Conjectures}.

\section{Analysis of almost extremal graphs}

For any $r\geq1$ denote $[r]=\{1,\ldots,r\}$. We denote by $H[X]$
the subgraph of $H$ induced by a set of vertices $X$ and by $H[X,Y]$
the bipartite subgraph of $H$ with parts $X,Y\subset V(H)$, $X\cap Y=\emptyset$.

In this section we deal with almost extremal cases, that is, with
triangulations of homology $(2r-1)$-manifolds that are close to $T_{r}(n)$.
These graphs fall into the class of \textit{$(\eta,r)$}\textit{\emph{-extremal
graphs introduced below.}}
\begin{defn}
\label{def:extremal}Let $0\leq\eta<1$ and $r\geq1$ be given. We
say that an $n$-vertex graph $H$ is \textit{$(\eta,r)$-extremal}\textit{\emph{
if the vertices of $H$ can be partitioned into sets $V_{0},V_{1},\ldots,V_{r}$
such that}}
\begin{enumerate}[label=(\emph{\alph*})]
\item \label{enu:Dext1}$|V_{0}|\le\frac{1}{30r^{r}}\eta n$ and $\lfloor(1-\frac{1}{30r}\eta)\frac{n}{r}\rfloor\le|V_{i}|\le\lceil(1+\frac{1}{30r}\eta)\frac{n}{r}\rceil$
for $i\in[r]$,
\item \label{enu:Dext2}$H[V_{i}]$ is triangle-free, for $i\in[r]$,
\item \label{enu:Dext3}$H[V_{i}]$ has maximum degree at most~2, for $i\in[r]$,
\item \label{enu:Dext4}for each $i,j\in[r]$, $i\neq j$, and any $v\in V_{i}$
we have $\deg_{H}(v,V_{j})\geq(1-\eta)|V_{j}|$,
\item \label{enu:Dext5}each vertex of $V_{0}$ is either of Type 1 or Type
2, where
\end{enumerate}
we say a vertex $v$ is of \emph{Type~1} if there exist two distinct
indices $g,h\in[r]$ such that $\deg_{H}(v,V_{g})\le2$ and $\deg_{H}(v,V_{h})\le(1-\frac{1}{2}\eta)|V_{h}|$,
and it is of \emph{Type~2} if there exist two distinct indices $g,h\in[r]$
such that $\deg_{H}(v,V_{g})\leq3r\eta|V_{g}|$ and $\deg_{H}(v,V_{h})\le3r\eta|V_{h}|$.
\end{defn}

\medskip{}

Since $\lfloor\frac{n}{r}\rfloor,\lceil\frac{n}{r}\rceil=\frac{n}{r}\pm1$,
for any $\eta>0$ fixed and $n$ large, it would be an inessential
change to require $(1-\frac{1}{30r}\eta)\frac{n}{r}\le|V_{i}|\le(1+\frac{1}{30r}\eta)\frac{n}{r}$
in~\ref{enu:Dext1} of Definition~\ref{def:extremal}. The floors
and ceilings in~\ref{enu:Dext1} are important to make the case $\eta=0$
meaningful, when $\lfloor\frac{n}{r}\rfloor$ and $\lceil\frac{n}{r}\rceil$
are feasible cardinalities of sets, while $\frac{n}{r}$ need not
be.

For a small $\eta>0$, graphs with the structure given by Definition~\ref{def:extremal}
resemble $J_{r}(n)$ up to some error. That is, we allow that a small
fraction of edges missing in each $H[V_{i},V_{j}]$, that the parts
are slightly unbalanced and we admit a small set of exceptional vertices
$V_{0}$. In the next definition we introduce a class of graphs that
resemble $J_{r}(n)$ even better.
\begin{defn}
\label{def:radical}We say that a graph is \emph{$r$-radical} if
it is $(0,r)$-extremal, and for each $i\in[r]$ every vertex of $H[V_{i}]$
has degree $2$.
\end{defn}

If $H$ is $r$-radical then $V_{0}=\emptyset$, each $V_{i}$ is
of size $\lfloor\frac{n}{r}\rfloor$ or $\lceil\frac{n}{r}\rceil$
for $i\in[r]$ and each graph $H[V_{i},V_{j}]$ is complete bipartite
for $i,j\in[r]$, $i\neq j$.\textit{\emph{ An $r$-radical graph
is $(\eta,r)$-extremal for any $0\leq\eta<1$. Note that if $H$
is any $n$-vertex $r$-radical graph then $F(H)=F(J_{r}(n))$ for
every clique function $F.$ }}
\begin{lem}
\label{lem:radicalmanifold}If $H$ is an $r$-radical graph with
$n$ vertices which triangulates a homology $(2r-1)$-manifold then
$H$ is isomorphic to $J_{r}(n)$.\end{lem}
\begin{proof}
For all $i=1,\ldots,r-1$ pick any edge in $H[V_{i}]$. The endpoints
of these $r-1$ edges form a clique of order $2r-2$ whose link is
$H[V_{r}]$. However, in a homology $(2r-1)$-manifold the link of
a face of size $2r-2$ is a homology $1$-sphere, that is a cycle.
It means that $H[V_{r}]$ is a cycle. The same argument shows that
all $H[V_{i}]$ are cycles and therefore $H$ is isomorphic to $J_{r}(n)$.
\end{proof}

The next lemma is used to find copies of $K_{3}^{r}$ in $(\eta,r)$-extremal
graphs.
\begin{lem}
\label{lem:bigk3}Fix $r\geq1$ and $0<\eta<\frac{1}{3r}$. Suppose
$H$ is a graph with $n\geq2r\eta^{-1}$ vertices and a partition
$V(H)=V_{0}\sqcup V_{1}\sqcup\cdots\sqcup V_{r}$ which satisfies
conditions~ \ref{enu:Dext1} and~\ref{enu:Dext4} of Definition~\ref{def:extremal}.
Let $w_{1},w_{2},w_{3}\in V_{1}$ be any three fixed vertices. For
$i\in\{2,\ldots,r\}$ let $A_{i}\subseteq V_{i}$ be sets with $|A_{i}|\geq3r\eta|V_{i}|$.
Then the subgraph of $H$ induced by $\{w_{1},w_{2},w_{3}\}\cup\bigcup_{i=2}^{r}A_{i}$
contains a $K_{3}^{r}$ with $3$ vertices in each part $V_{i}$,
$i\in[r]$.\end{lem}
\begin{proof}
We will construct by induction $3$-element subsets $\{w_{1}^{i},w_{2}^{i},w_{3}^{i}\}\subseteq A_{i}$
such that for each $l\in[r]$ the subgraph of $H$ induced by $w_{j}^{i}$
with $j=1,2,3$ and $i=1,\ldots,l$ contains $K_{3}^{l}$. For $l=r$
this proves the lemma. When $i=1$ the vertices $w_{j}^{1}=w_{j}$
are already given.

Suppose we have constructed the vertices $\{w_{1}^{i},w_{2}^{i},w_{3}^{i}\}_{i=1}^{l}$
for some $l\leq r-1$. By condition~\ref{enu:Dext4} the common neighborhood
$N_{l+1}$ of these $3l$ vertices satisfies $|N_{l+1}\cap V_{l+1}|\geq(1-3l\eta)|V_{l+1}|$.
It follows that 
\begin{eqnarray*}
|A_{l+1}\cap N_{l+1}| & \geq & |A_{l+1}|-|V_{l+1}\setminus N_{l+1}|\geq3r\eta|V_{l+1}|-3l\eta|V_{l+1}|\\
 & \geq & 3\eta|V_{l+1}|\geq3\eta\frac{n}{2r}\geq3,
\end{eqnarray*}
where the last line uses condition~\ref{enu:Dext1} of Definition~\ref{def:extremal}
and the bound $n\geq2r\eta^{-1}$. It means that we can pick three
distinct vertices $w_{1}^{l+1},w_{2}^{l+1},w_{3}^{l+1}\in A_{l+1}\cap N_{l+1}$
and the induction step is complete.
\end{proof}

We can now prove that graphs triangulating homology $(2r-1)$-manifolds
are $(\eta,r)$-extremal whenever they are sufficiently close to $T_{r}(n)$.
\begin{lem}
\label{lem:obtainextremal-1}For every $r\geq2$ and $0<\eta<\frac{1}{7r}$
set $\epsilon=\frac{\eta^{2}}{120r^{r+3}}$. If a graph $H$ with
$n\geq2r\eta^{-1}$ vertices triangulates a homology $(2r-1)$-manifold
and $H$ is $\epsilon$-close to $T_{r}(n)$ then $H$ is $(\eta,r)$-extremal.\end{lem}
\begin{proof}
If $H$ is $\epsilon$-close to $T_{r}(n)$ then the vertices of $H$
can be partitioned into $r$ sets $X_{1},\ldots,X_{r}$, each of size
$\lfloor\frac{n}{r}\rfloor$ or $\lceil\frac{n}{r}\rceil$, such that
\[
\sum_{i<j}\overline{e}_{2}(H[X_{i},X_{j}])+\sum_{i}e_{2}(H[X_{i}])\leq\epsilon n^{2}.
\]
Here $\overline{e}_{2}(H[X_{i},X_{j}])$ is the number of edges missing
between $X_{i}$ and $X_{j}$, that is $\overline{e}_{2}(H[X_{i},X_{j}])=|X_{i}|\cdot|X_{j}|-e_{2}(H[X_{i},X_{j}])$.
For every $i\neq j$ let 
\[
X_{i,j}=\{v\in X_{i}\colon\deg(v,X_{j})\leq(1-\frac{2}{3}\eta)|X_{j}|\}.
\]
Every vertex in $X_{i,j}$ contributes to the number of missing edges
$\overline{e}_{2}(H[X_{i},X_{j}])$ as follows 
\[
\epsilon n^{2}\geq\overline{e}_{2}(H[X_{i},X_{j}])\geq\sum_{v\in X_{i,j}}(|X_{j}|-\deg(v,X_{j}))\geq\frac{2}{3}\eta|X_{j}|\cdot|X_{i,j}|\geq\frac{1}{2}\eta\frac{n}{r}\cdot|X_{i,j}|,
\]
 hence $|X_{i,j}|\leq2\epsilon r\eta^{-1}n$. Consider a new partition
$V(H)=Y_{0}\sqcup Y_{1}\sqcup\ldots\sqcup Y_{r}$, 
\[
Y_{0}=\bigcup_{i\neq j}X_{i,j},\quad Y_{i}=X_{i}\setminus Y_{0}\ \mathrm{\mathrm{for}}\ i\in[r].
\]
 We have $|Y_{0}|\leq r^{2}\cdot2\epsilon r\eta^{-1}n=\frac{1}{60r^{r}}\eta n$
and, for $i\in[r]$, 
\[
\lceil\frac{n}{r}\rceil\geq|X_{i}|\geq|Y_{i}|\geq|X_{i}|-|Y_{0}|\geq\lfloor\frac{n}{r}\rfloor-\frac{1}{60r^{r}}\eta n\ge(1-\tfrac{1}{30r}\eta)\frac{n}{r}.
\]
 By definition, for every vertex $v\in Y_{0}$ there exists an index
$j\in[r]$ such that $\deg(v,X_{j})\leq(1-\frac{2}{3}\eta)|X_{j}|$.
Let $Z_{j}\subset Y_{0}$ consists of those vertices for which $j$
is the only such index, formally: 
\[
Z_{j}=\left\{ v\in Y_{0}\colon\deg(v,X_{k})\leq(1-\tfrac{2}{3}\eta)|X_{k}|\ \mbox{\ensuremath{\mathrm{iff}}}\ k=j\ \mathrm{for}\ k\in[r]\right\} .
\]
 We now define the final partition of $V(H)$ as 
\[
V_{0}=Y_{0}\setminus\bigcup_{j}Z_{j},\quad V_{i}=Y_{i}\cup Z_{i}\ \mathrm{for}\ i\in[r].
\]
We claim that the partition $V(H)=V_{0}\sqcup V_{1}\sqcup\ldots\sqcup V_{r}$
witnesses $(\eta,r)$-extremality of $H$. We have $|V_{0}|\leq|Y_{0}|\leq\frac{1}{30r^{r}}\eta n$
and $|V_{i}|\geq|Y_{i}|\geq(1-\frac{1}{30r}\eta)\frac{n}{r}$ for
$i\in[r]$. Moreover, 
\[
|V_{i}|=|Y_{i}|+|Z_{i}|\leq|Y_{i}|+|Y_{0}|\leq\lceil\tfrac{n}{r}\rceil+\tfrac{1}{60r^{r}}\eta n\leq(1+\tfrac{1}{30r}\eta)\frac{n}{r}.
\]
That proves condition~\ref{enu:Dext1} of Definition~\ref{def:extremal}.
Next we verify condition~\ref{enu:Dext4}. Pick any vertex $v\in V_{i}$,
$i\in[r]$. Regardless of whether $v\in Y_{i}$ or $v\in Z_{i}$ we
have that $\deg(v,X_{j})\geq(1-\frac{2}{3}\eta)|X_{j}|$ for all $j\neq i$.
That yields 
\begin{eqnarray*}
\deg(v,V_{j}) & \geq & \deg(v,Y_{j})\geq\deg(v,X_{j})-|Y_{0}|\geq(1-\tfrac{2}{3}\eta)|X_{j}|-|Y_{0}|\\
 & \geq & (1-\tfrac{2}{3}\eta-\tfrac{\eta}{30r^{r-1}})\frac{n}{r}\geq(1-\tfrac{2}{3}\eta-\tfrac{\eta}{30r^{r-1}})(1+\tfrac{1}{30r}\eta)^{-1}|V_{j}|\\
 & \geq & (1-\eta)|V_{j}|.
\end{eqnarray*}

To prove property~\ref{enu:Dext2}, suppose, without loss of generality,
that $H[V_{1}]$ contains a triangle $t=\{w_{1},w_{2},w_{3}\}$. For
$i=2,\ldots,r$ let $A_{i}=N_{H}(w_{1})\cap N_{H}(w_{2})\cap N_{H}(w_{3})\cap V_{i}.$
By the already proven property~\ref{enu:Dext4} we have $|A_{i}|\geq(1-3\eta)|V_{i}|\geq3\eta r|V_{i}|$
(the last inequality uses $\eta<\frac{1}{7r}$). Since $n\geq2r\eta^{-1}$
Lemma~\ref{lem:bigk3} now yields that the link $\mathrm{lk}_{H}(t)$
contains $K_{3}^{r-1}$ as a subgraph. This is a contradiction to
Lemma~\ref{lem:vankampen}, since $\mathrm{lk}_{H}(t)$ triangulates
a homology sphere of dimension $2r-1-3=2(r-2).$

Similarly, to prove~\ref{enu:Dext3}, suppose $v\in V_{1}$ has three
distinct neighbors $w_{1},w_{2},w_{3}\in V_{1}$. Applying Lemma~\ref{lem:bigk3}
with $w_{1},w_{2},w_{3}$ and $A_{i}=N_{H}(v)\cap V_{i}$ for $i=2,\ldots,r$,
where $|A_{i}|\geq(1-\eta)|V_{i}|\geq3\eta r|V_{i}|$, we get that
$\mathrm{lk}_{H}(v)$ contains a $K_{3}^{r}$. This contradicts the
fact that $\mathrm{lk}_{H}(v)$ triangulates a homology $2(r-1)$-sphere.

We now turn to verifying~\ref{enu:Dext5}. Let us start with an auxiliary
claim. 
\begin{claim*}
Let $v\in V_{0}$ be any vertex and suppose $j\in[r]$ is any index
such that $\deg(v,X_{j})\leq(1-\frac{2}{3}\eta)|X_{j}|$. Then $\deg(v,V_{j})\le(1-\frac{1}{2}\eta)|V_{j}|$. \end{claim*}
\begin{proof}
We have 
\begin{eqnarray*}
\deg(v,V_{j}) & \leq & \deg(v,Y_{j})+|Z_{j}|\leq\deg(v,X_{j})+|Y_{0}|\\
 & \leq & (1-\tfrac{2}{3}\eta)\lceil\tfrac{n}{r}\rceil+\tfrac{1}{60r^{r}}\eta n\leq(1-\tfrac{2}{3}\eta+\tfrac{1}{15r^{r-1}}\eta)\tfrac{n}{r}\\
 & \leq & (1-\tfrac{2}{3}\eta+\tfrac{1}{15r^{r-1}}\eta)(1-\tfrac{1}{30r}\eta)^{-1}|V_{j}|\leq(1-\tfrac{1}{2}\eta)|V_{j}|.
\end{eqnarray*}

\end{proof}
Now suppose that some vertex $v\in V_{0}$ is not of Type~2. Then,
without loss of generality, $\deg(v,V_{i})>3\eta r|V_{i}|$ for $i=2,\ldots,r$.
Suppose that $\deg(v,V_{1})\geq3$ and let $w_{1},w_{2},w_{3}\in N_{H}(v)\cap V_{1}$
be three distinct vertices. We already proved properties~\ref{enu:Dext1}
and~\ref{enu:Dext4}, so we can apply Lemma~\ref{lem:bigk3} with
$w_{1},w_{2},w_{3}$ and $A_{i}=N_{H}(v)\cap V_{i}$ for $i=2,\ldots,r$
to conclude that $\mathrm{lk}_{H}(v)$ contains a $K_{3}^{r}$, a
contradiction to Lemma~\ref{lem:vankampen}. Therefore, $\deg(v,V_{1})\leq2$.
By the definition of $V_{0}$, there exist an index $j\neq1$ such
$\deg(v,X_{j})\leq(1-\frac{2}{3}\eta)|X_{j}|$. The above Claim then
gives that $\deg(v,V_{j})\leq(1-\frac{1}{2}\eta)|V_{j}|.$ This proves
that $v$ is of Type~1. Condition~\ref{enu:Dext5} follows.

This completes the proof of the lemma.
\end{proof}
Our last lemma says that for among $(\eta,r)$-extremal graphs, the
graph $J_{r}(n)$ maximizes any clique function of order up to $r$
(for sufficiently large $n$). Note that in this part of the proof
we do not assume that $H$ triangulates a homology manifold.
\begin{lem}
\label{lem:optimizeExtremal}Let $r\geq2$ and let $F$ be a clique
function of order $k$, $2\le k\le r$. Set $\eta=\frac{1}{14r^{r}}$.
Then there exists a number $m_{0}$ such that the following holds.
If $H$ is an $(\eta,r)$-extremal graph with $n\geq m_{0}$ vertices
then $F(H)\le F(J_{r}(n))$ and equality is attained only when $H$
is $r$-radical.\end{lem}
\begin{proof}
Let the clique function be $F(G)=c_{k}e_{k}(G)+c_{k-1}e_{k-1}(G)+\ldots+c_{1}e_{1}(G)+c_{0}$.
The value of $m_{0}$ will be chosen during the proof in such a way
that (\ref{eq:maxAddEdge}), (\ref{eq:maxType1}), (\ref{eq:maxType2}),
(\ref{eq:maxCycles}) and (\ref{eq:maxEqualParts}) are satisfied
for all $n\geq m_{0}$. Among all $(\eta,r)$-extremal graphs with
$n$ vertices, let us consider a graph $H$ that maximizes $F(H)$.
We will prove that $H$ is $r$-radical. The proof proceeds in five
claims in which we derive the structural properties asserted by $r$-radicality.
More specifically, we will prove that if $H$ violated any of the
conditions of Definition~\ref{def:radical} then we could modify
$H$ locally so that the resulting graph were still $(\eta,r)$-extremal,
yet with a strictly greater value of the clique function $F$.
\begin{claim*}
For each $i,j\in[r]$, $i\neq j$, the bipartite graph $H[V_{i},V_{j}]$
is complete.\end{claim*}
\begin{proof}
Suppose for a contradiction and without loss of generality that there
exist vertices $v_{1}\in V_{1}$ and $v_{2}\in V_{2}$ that do not
form an edge. Let us now add that edge to $H$. Observe that the modified
graph $H'$ is still $(\eta,r)$-extremal. We will now find a lower
bound for the number of cliques in $H'$ which contain the edge $v_{1}v_{2}.$
By condition~\ref{enu:Dext4}, $v_{1}$ and $v_{2}$ are both adjacent
to at least $(1-2\eta)|V_{3}|\ge(1-3\eta)\frac{n}{r}$ vertices $v_{3}$
in $V_{3}$. In general, given vertices $v_{1}\in V_{1},v_{2}\in V_{2},\ldots,v_{\ell}\in V_{\ell}$
there are at least $(1-\ell\eta)|V_{\ell+1}|\ge(1-(\ell+1)\eta)\frac{n}{r}$
vertices $v_{\ell+1}$ in $V_{l+1}$ adjacent to each of $v_{1},v_{2},\ldots,v_{\ell}$.
This sequential extension gives at least $((1-k\eta)\frac{n}{r})^{k-2}>(\frac{1}{2}\cdot\frac{n}{r})^{k-2}$
many $k$-cliques containing both $v_{1}$ and $v_{2}$ in $H'$.
For each $t=2,\ldots,k-1$, the number of $t$-cliques increased by
at most $n^{t-2}$, and the number of vertices did not change. So,
in total, 
\begin{equation}
F(H')-F(H)\ge c_{k}\left(\frac{n}{2r}\right)^{k-2}-\sum_{t=2}^{k-1}|c_{t}|n^{t-2}>0\;,\label{eq:maxAddEdge}
\end{equation}
since the coefficients $c_{t}$ are fixed and $n\geq m_{0}$ is large
enough. This is a contradiction to the assumption that $H$ maximizes
$F$.\end{proof}
\begin{claim*}
The set $V_{0}$ does not contain any Type~1 vertex.\end{claim*}
\begin{proof}
Suppose that $v\in V_{0}$ is a Type~1 vertex. Let $g$ and $h$
be the two indices as in the definition of Type~1 in Definition~\ref{def:extremal}.
Since the average size of the sets $V_{i}$, $i\in[r]$, is $\frac{n-|V_{0}|}{r}$,
there is an index $j\in[r]$ so that $|V_{j}|<\frac{n}{r}$. We construct
a new graph $H'$ by deleting $v$ (and its incident edges) from $V_{0}$
and introducing a new vertex $w$ into the set $V_{j}$. We make $w$
adjacent to all the vertices in $\bigcup_{i\in[r]\setminus{j}}V_{i}$,
and to no other. The modified graph $H'$ is $(\eta,r)$-extremal. 

The vertex $w$ is contained in at least $\binom{r-1}{k-1}\lfloor(1-\frac{1}{30r}\eta)\frac{n}{r}\rfloor^{k-1}$
many $k$-cliques in $H'$. Indeed, we can choose an arbitrary $(k-1)$-element
set $\{p_{1},p_{2},\ldots,p_{k-1}\}\subset[r]\setminus j$, and this
choice gives us at least $\lfloor(1-\frac{1}{30r}\eta)\frac{n}{r}\rfloor^{k-1}$
choices of vertices $w_{1}\in V_{p_{1}},\ldots,w_{k-1}\in V_{p_{k-1}}$.
By the previous Claim, for any such choice $\{w,w_{1},\ldots,w_{k-1}\}$
is a clique.

Let us now upper-bound the number of $k$-cliques in $H$ containing
$v$. The number of cliques containing $v$ and some other vertex
of $V_{0}$ is at most $|V_{0}|\cdot n^{k-2}\leq\frac{1}{30r^{r}}\eta n^{k-1}$.
The number of $k$-cliques through $v$ and through a vertex from
the set $V_{g}$ is at most $2n^{k-2}$ by the definition of Type~1.
By Definition~\ref{def:extremal}\ref{enu:Dext3} and \ref{enu:Dext4},
if $k\geq3$ the number of cliques containing $v$ and at least two
vertices from a fixed $V_{i}$, $i\in[r]$, is at most $e_{2}(H[V_{i}])\cdot n^{k-3}\leq|V_{i}|\cdot n^{k-3}$.
Therefore the number of $k$-cliques that touch some of the sets $V_{i}$
in at least two vertices is upper bounded by $n^{k-2}$. It remains
to upper-bound the number of $k$-cliques in $H$ through $v$ that
contain no vertex from $(V_{0}\setminus\{v\})\cup V_{g}$, and that
intersect each of the sets $V_{i}$ in at most one vertex. Trivially,
this number is at most $\binom{r-1}{k-1}\lceil(1+\frac{1}{30r}\eta)\frac{n}{r}\rceil^{k-1}$.
However, the fact that $\deg(v,V_{h})\le(1-\frac{1}{2}\eta)|V_{h}|$
allows us to refine this upper-bound to 
\begin{eqnarray*}
 &  & \binom{r-2}{k-1}\lceil(1+\tfrac{1}{30r}\eta)\tfrac{n}{r}\rceil^{k-1}+\binom{r-2}{k-2}\lceil(1+\tfrac{1}{30r}\eta)\tfrac{n}{r}\rceil^{k-2}(1-\tfrac{1}{2}\eta)\lceil(1+\tfrac{1}{30r}\eta)\tfrac{n}{r}\rceil\\
 & = & \lceil(1+\tfrac{1}{30r}\eta)\tfrac{n}{r}\rceil^{k-1}\left(\binom{r-1}{k-1}-\tfrac{1}{2}\eta\cdot\binom{r-2}{k-2}\right)\\
 & = & \binom{r-1}{k-1}\lceil(1+\tfrac{1}{30r}\eta)\tfrac{n}{r}\rceil^{k-1}(1-\tfrac{\eta}{2}\cdot\tfrac{k-1}{r-1})\ .
\end{eqnarray*}
Putting these bounds together, we get 
\begin{eqnarray*}
 &  & e_{k}(H')-e_{k}(H)\\
 & \ge & \binom{r-1}{k-1}\lfloor(1-\tfrac{1}{30r}\eta)\tfrac{n}{r}\rfloor^{k-1}-\binom{r-1}{k-1}\lceil(1+\tfrac{1}{30r}\eta)\tfrac{n}{r}\rceil^{k-1}(1-\tfrac{\eta}{2}\cdot\tfrac{k-1}{r-1})\\
 &  & \qquad\qquad\qquad\qquad-\tfrac{1}{30r^{r}}\eta n^{k-1}-3n^{k-2}\ .
\end{eqnarray*}
Using the inequality $\lfloor(1-\frac{1}{30r}\eta)\frac{n}{r}\rfloor>\lceil(1+\frac{1}{30r}\eta)\frac{n}{r}\rceil(1-\frac{1}{10r}\eta)$
we can write 
\begin{eqnarray*}
 &  & e_{k}(H')-e_{k}(H)\\
 & > & \binom{r-1}{k-1}\lceil(1+\tfrac{1}{30r}\eta)\tfrac{n}{r}\rceil^{k-1}[(1-\tfrac{1}{10r}\eta)^{k-1}-1+\tfrac{\eta}{2}\cdot\tfrac{k-1}{r-1}]\\
 &  & \qquad\qquad\qquad\qquad-\tfrac{1}{30r^{r}}\eta n^{k-1}-3n^{k-2}\ .
\end{eqnarray*}
By Bernoulli's inequality the coefficient in the square brackets is
at least 
\[
1-\frac{k-1}{10r}\eta-1+\frac{\eta}{2}\cdot\frac{k-1}{r-1}>-\frac{k-1}{10r}\eta+\frac{k-1}{2r}\eta=\frac{2}{5}\cdot\frac{k-1}{r}\eta\ .
\]
That gives 
\begin{eqnarray*}
e_{k}(H')-e_{k}(H) & > & {r-1 \choose k-1}\frac{n^{k-1}}{r^{k-1}}\cdot\frac{2}{5}\cdot\frac{k-1}{r}\eta-\frac{1}{30r^{r}}\eta n^{k-1}-3n^{k-2}\\
 & = & \frac{1}{r^{r}}\eta n^{k-1}\left(\tfrac{2}{5}{r-1 \choose k-1}(k-1)r^{r-k}-\tfrac{1}{30}\right)-3n^{k-2}\\
 & \geq & \frac{1}{r^{r}}\eta n^{k-1}(\tfrac{2}{5}-\tfrac{1}{30})-3n^{k-2}=\tfrac{11}{30r^{r}}\eta n^{k-1}-3n^{k-2}\ .
\end{eqnarray*}

The number of cliques of size $t$ changed by at most $n^{t-1}$ for
$t=2,\ldots,k-1$. That implies 
\begin{equation}
F(H')-F(H)>\tfrac{11}{30r^{r}}\eta c_{k}n^{k-1}-3c_{k}n^{k-2}-\sum_{t=2}^{k-1}|c_{t}|n^{t-1}>0\label{eq:maxType1}
\end{equation}
 since $n\geq m_{0}$ is sufficiently large. That contradicts the
maximality of $H$ and proves the claim.\end{proof}
\begin{claim*}
The set $V_{0}$ does not contain any Type~2 vertex. \end{claim*}
\begin{proof}
We proceed similarly as in the previous case. Suppose that $v\in V_{0}$
is a Type~2 vertex. Let $g$ and $h$ be the two indices as in Definition~\ref{def:extremal}.
We delete $v$ from $V_{0}$ and introduce a new vertex $w$ in some
set $V_{j}$, $j\in[r]$ with $|V_{j}|<\frac{n}{r}$ which we make
adjacent to all the vertices in $\bigcup_{i\in[r]\setminus\{j\}}V_{i}$,
and to no other. Let $H'$ be the resulting $(\eta,r)$-extremal graph.
As before, the new vertex $w$ belongs to at least $\binom{r-1}{k-1}\lfloor(1-\frac{1}{30r}\eta)\tfrac{n}{r}\rfloor^{k-1}$
cliques of size $k$ in $H'$.

Next we upper-bound the number of cliques containing $v$ in $H$.
The number of $k$-cliques through $v$ and through a vertex from
the set $V_{g}\cup V_{h}\cup(V_{0}\setminus\{v\})$ is at most $(|V_{g}\cap N_{H}(v)|+|V_{h}\cap N_{H}(v)|+|V_{0}|)n^{k-2}\leq7\eta n^{k-1}$
by the definition of Type~2. The number of $k$-cliques through $v$
that touch at least two vertices in some $V_{i}$ is at most $n^{k-2}$,
as in the previous claim. Last, the number of $k$-cliques through
$v$ that do not intersect $V_{g}\cup V_{h}\cup(V_{0}\setminus\{v\})$
and contain at most one vertex from each $V_{i}$ is upper-bounded
by 
\begin{eqnarray*}
\binom{r-2}{k-1}\lceil(1+\tfrac{1}{30r}\eta)\tfrac{n}{r}\rceil^{k-1} & = & \frac{r-k}{r-1}\binom{r-1}{k-1}\lceil(1+\tfrac{1}{30r}\eta)\tfrac{n}{r}\rceil^{k-1}
\end{eqnarray*}
(in particular it must be $0$ when $k=r$). Proceeding as in the
proof of the previous claim we get 
\begin{eqnarray*}
 &  & e_{k}(H')-e_{k}(H)\\
 & \geq & \binom{r-1}{k-1}\lfloor(1-\tfrac{1}{30r}\eta)\tfrac{n}{r}\rfloor^{k-1}-\tfrac{r-k}{r-1}\binom{r-1}{k-1}\lceil(1+\tfrac{1}{30r}\eta)\tfrac{n}{r}\rceil^{k-1}-7\eta n^{k-1}-n^{k-2}\\
 & \geq & \binom{r-1}{k-1}\lceil(1+\tfrac{1}{30r}\eta)\tfrac{n}{r}\rceil^{k-1}[(1-\tfrac{1}{10r}\eta)^{k-1}-\tfrac{r-k}{r-1}]-7\eta n^{k-1}-n^{k-2}.
\end{eqnarray*}
The expression in the square brackets is at least 
\[
1-\frac{k-1}{10r}\eta-1+\frac{k-1}{r-1}>(k-1)\left(\frac{1}{r}-\frac{1}{10r}\eta\right)>\frac{9(k-1)}{10r}\geq\frac{9}{10r}\ .
\]
Hence we get 
\begin{eqnarray*}
e_{k}(H')-e_{k}(H) & \geq & n^{k-1}\left(\frac{9}{10r^{k}}\binom{r-1}{k-1}-7\eta\right)-n^{k-2}>\frac{1}{3r^{r}}n^{k-1}-n^{k-2}\ ,
\end{eqnarray*}
where we used $7\eta\leq\frac{1}{2r^{r}}$, and finally 
\begin{equation}
F(H')-F(H)>\frac{1}{3r^{r}}c_{k}n^{k-1}-c_{k}n^{k-2}-\sum_{t=2}^{k-1}|c_{t}|n^{t-1}>0\;,\label{eq:maxType2}
\end{equation}
because $n\geq m_{0}$.
\end{proof}
Thus, by the three claims above, the vertex set of $H$ is partitioned
into sets $V_{1},\ldots,V_{r}$, all pairs of which form complete
bipartite graphs. Recall that the graphs $H[V_{i}]$ are triangle-free
and of maximum degree at most~2. 
\begin{claim*}
For each $i\in[r]$, we have $e_{2}(H[V_{i}])=|V_{i}|$.\end{claim*}
\begin{proof}
The condition that the maximum degree of $H[V_{i}]$ is at most~2
implies that $e_{2}(H[V_{i}])\le|V_{i}|$. Suppose now that $e_{2}(H[V_{i}])<|V_{i}|$.
We replace the subgraph $H[V_{i}]$ with the graph consisting of a
path with $e_{2}(H[V_{i}])$ edges followed by $|V_{i}|-e_{2}(H[V_{i}])-1$
isolated vertices. Let $H'$ be the resulting graph. Note that $H'$
is $(\eta,r)$-extremal, and since $H[V_{i}]$ was triangle-free we
have $e_{\ell}(H')=e_{\ell}(H)$ for all $\ell$. Next, we create
$H''$ by adding one edge to $H'[V_{i}]$, so that we get a longer
path or a cycle. We still have that $H''$ is $(\eta,r)$-extremal. 

The number of $k$-cliques increased from $H'$ to $H''$ by at least
$\binom{r-1}{k-2}\lfloor(1-\frac{1}{30r}\eta)\tfrac{n}{r}\rfloor^{k-2}\geq(\frac{n}{2r})^{k-2}$.
At the same time, the total number of cliques of order $t=2,\ldots,k-1$
increased by at most $n^{t-2}$. Hence 
\begin{equation}
F(H'')-F(H)=F(H'')-F(H')\ge\left(\frac{n}{2r}\right){}^{k-2}-\sum_{t=2}^{k-1}|c_{t}|n^{t-2}>0\label{eq:maxCycles}
\end{equation}
for $n\geq m_{0}$, a contradiction to the supposed maximality of
$H$.\end{proof}
\begin{claim*}
For each $i,j\in[r]$, $i\neq j$, we have $|V_{i}|-1\le|V_{j}|\le|V_{i}|+1$.\end{claim*}
\begin{proof}
Consider the class of all $(\eta,r)$-extremal graphs $G$ with sufficiently
many vertices partitioned into sets $V(G)=V_{0}(G)\sqcup V_{1}(G)\sqcup\ldots\sqcup V_{r}(G)$
satisfying all the previous claims, i.e. $V_{0}(G)=\emptyset$, $G[V_{i}(G),V_{j}(G)]$
is complete bipartite for $i,j\in[r],\ i\neq j$, each $G[V_{i}(G)]$
is triangle-free and $e_{2}(G[V_{i}(G)])=|V_{i}(G)|$ for $i\in[r]$.
Let $\sigma_{j}(x_{1},\ldots,x_{r})=\sum_{1\leq i_{1}<\cdots<i_{j}\leq r}x_{i_{1}}\cdots x_{i_{j}}$
denote the $j$-th elementary symmetric polynomial in $r$ variables.
For a graph $G$ in the above class we have 
\begin{equation}
e_{\ell}(G)=\sum_{j=0}^{\ell}\tbinom{j}{\ell-j}\;\sigma_{j}\left(|V_{1}(G)|,\ldots,|V_{r}(G)|\right),\quad\ell=0,\ldots,2r.\label{eq:MichaelJacksonIsMyBigHero}
\end{equation}
To see this, let us write $n(i)=|V_{i}(G)|$. Let us label arbitrarily
the vertices in each set $V_{i}(G)$ as $v_{1}^{i},v_{2}^{i},\ldots,v_{n(i)}^{i}$,
and the edges in $V_{i}(G)$ as $e_{1}^{i},e_{2}^{i},\ldots,e_{n(i)}^{i}$.
Consider any $j$ with $0\le j\le\ell$. Then it is enough to argue
that the summand $\tbinom{j}{\ell-j}\:\sigma_{j}\left(|V_{1}(G)|,\ldots,|V_{r}(G)|\right)$
counts the $\ell$-cliques that intersect exactly $j$ of the sets
$V_{1}(G),\ldots,V_{r}(G)$. Note that such a clique touches exactly
$\ell-j$ sets in an edge and the remaining ones in one vertex. The
quantity $\sigma_{j}\left(|V_{1}(G)|,\ldots,|V_{r}(G)|\right)$ counts
the number of ways of selecting $j$ vertices in $V_{1}(G)\sqcup\ldots\sqcup V_{r}(G)$
so that each set $V_{i}(G)$ contains at most one selected vertex.
The choices represented by the binomial coefficient $\tbinom{j}{\ell-j}$
then indicate which of the selected vertices $v_{h}^{i}$ should be
replaced by edges $e_{h}^{i}$ (the indices in $v_{h}^{i}$ and $e_{h}^{i}$
are the same) so to get one desired $\ell$-clique. This establishes~(\ref{eq:MichaelJacksonIsMyBigHero}).

It follows that there are constants $c_{k}',\ldots,c_{0}'$ depending
only on $F$ and $r$ such that $c_{k}'=c_{k}>0$ and $F(G)=\sum_{i=0}^{k}c_{i}'\sigma_{i}(|V_{1}(G)|,\ldots,|V_{r}(G)|)$
for all graphs $G$ in the class. Now suppose, without loss of generality,
that in the maximizer $H$ we have $|V_{1}(H)|-|V_{2}(H)|\geq2$.
Take any graph $H'$ in the same class with parts of size $(|V_{1}|-1,|V_{2}|+1,|V_{3}|,\ldots,|V_{r}|)$.
For any numbers $x_{1},\ldots,x_{r}$ have 
\[
\sigma_{j}(x_{1}-1,x_{2}+1,x_{3},\ldots,x_{r})-\sigma_{j}(x_{1},\ldots,x_{r})=(x_{1}-x_{2}-1)\sigma_{j-2}(x_{3},\ldots,x_{r})
\]
 and hence 
\begin{eqnarray}
F(H')-F(H) & = & (|V_{1}|-|V_{2}|-1)\sum_{i=2}^{k}c_{i}'\sigma_{i-2}(|V_{3}|,\ldots,|V_{r}|)\nonumber \\
 & \geq & c_{k}'{r-2 \choose k-2}\left(\frac{n}{2r}\right){}^{k-2}-\sum_{t=2}^{k-1}|c_{t}'|{r-2 \choose t-2}\lceil(1+\tfrac{1}{30r}\eta)\frac{n}{r}\rceil^{t-2}>0\label{eq:maxEqualParts}
\end{eqnarray}
for sufficiently large $n\geq m_{0}$, again a contradiction to the
maximality of $H$. 
\end{proof}
The claims above clearly prove the lemma.
\end{proof}

\section{Proof of the main theorem\label{sec:Proof}}

We can now prove Theorem~\ref{thm:MainGraphTheory}. Fix $r\geq k\geq2$
and a clique function $F(G)=\sum_{i=0}^{k}c_{i}e_{i}(G)$ with $c_{k}>0$. 

Let $\eta=\frac{1}{14r^{r}}$ and $m_{0}$ be the constants provided
by Lemma~\ref{lem:optimizeExtremal} given $r$ and $F$. Let $\epsilon=\frac{\eta^{2}}{120r^{r+3}}$
be the constant provided by Lemma~\ref{lem:obtainextremal-1} given
$r$ and $\eta$.

Let $\delta$ be the constant from Theorem~\ref{thm:ErdSimStability}
for input parameters $r$ and $\frac{1}{2}\epsilon$. Define $\alpha=\min\{\frac{1}{4}{r \choose k}\frac{1}{r^{k}}\delta,\frac{1}{2}\epsilon\}$
and let $\beta$ be the constant from Theorem~\ref{thm:Removal}
for input $r$ and $\alpha$. Let $m_{1}$ be such that for $n\geq m_{1}$
we have $(1-\frac{1}{4}\delta)F(T_{r}(n))<F(J_{r}(n))$. Let $m_{2}$
be such that for each $n\geq m_{2}$ and each $n$-vertex graph $G$
the condition $F(G)>(1-\frac{1}{4}\delta)F(T_{r}(n))$ implies $e_{k}(G)>(1-\frac{1}{2}\delta)e_{k}(T_{r}(n))$.
The existence of $m_{1}$ and $m_{2}$ follows by observing that for
$n$-vertex graphs $G$ we have $F(G)=c_{k}e_{k}(G)+O(n^{k-1})$ and
moreover $F(T_{r}(n))=c_{k}{r \choose k}\left(\frac{n}{r}\right)^{k}+O(n^{k-1})$,
$F(J_{r}(n))=c_{k}{r \choose k}\left(\frac{n}{r}\right)^{k}+O(n^{k-1})$.
Finally let $C_{r}$ be the constant from Lemma~\ref{lem:middleDS}.
We claim that Theorem~\ref{thm:MainGraphTheory} holds for $n_{0}=\max\{m_{0},m_{1},m_{2},C_{r}\beta^{-1},2r\eta^{-1}\}$.

Suppose $H$ is any graph with $n\geq n_{0}$ vertices which triangulates
a homology $(2r-1)$-manifold. First, suppose that $F(H)\leq(1-\frac{1}{4}\delta)F(T_{r}(n))$.
Since $n\geq m_{1}$ this implies $F(H)<F(J_{r}(n))$, and the result
is proved (in that case, equality is impossible).

That leaves us with the case $F(H)>(1-\frac{1}{4}\delta)F(T_{r}(n))$.
Since $n\geq m_{2}$ we get $e_{k}(H)>(1-\frac{1}{2}\delta)e_{k}(T_{r}(n))$.
By Lemma~\ref{lem:middleDS} we have $e_{r+1}(H)\leq C_{r}n^{r}\leq\beta n^{r+1}$.
Theorem~\ref{thm:Removal} now shows that we can remove at most $\alpha n^{2}$
edges from $H$ to obtain a $K_{r+1}$-free subgraph $G$ with the
same vertex set. The removal of one edge destroys at most $n^{k-2}$
cliques of size $k$, therefore 
\[
e_{k}(G)\geq e_{k}(H)-\alpha n^{k}\geq(1-\tfrac{1}{2}\delta)e_{k}(T_{r}(n))-\frac{1}{4}\delta{r \choose k}\frac{1}{r^{k}}n^{k}\geq(1-\delta)e_{k}(T_{r}(n)),
\]
where in the last step we used $e_{k}(T_{r}(n))\geq\frac{1}{2}{r \choose k}\frac{n^{k}}{r^{k}}$.
Theorem~\ref{thm:Zykov} now gives $e_{2}(G)\geq(1-\delta)e_{2}(T_{r}(n))$.
By Theorem~\ref{thm:ErdSimStability} the graph $G$ is $\frac{1}{2}\epsilon$-close
to $T_{r}(n)$. Since $H$ arises from $G$ by adding at most $\alpha n^{2}\leq\frac{1}{2}\epsilon n^{2}$
edges, we conclude that $H$ is $\epsilon$-close to $T_{r}(n)$.
From Lemma~\ref{lem:obtainextremal-1}, we have that $H$ is $(\eta,r)$-extremal.
As $n\geq m_{0}$, Lemma~\ref{lem:optimizeExtremal} now shows that
$F(H)\leq F(J_{r}(n))$. That ends the proof of the inequality.

If $F(H)=F(J_{r}(n))$ then by Lemma~\ref{lem:optimizeExtremal}
the graph $H$ is $r$-radical. Since $H$ triangulates a homology
$(2r-1)$-manifold, Lemma~\ref{lem:radicalmanifold} yields that
$H$ is isomorphic to $J_{r}(n)$. That proves the uniqueness part.

\section{Conjectures\label{sec:Conjectures}}

The number $n_{0}$ coming from our proof of Theorem~\ref{thm:Main}
is enormous. Indeed, the bottleneck is the use of the Removal lemma
(Theorem \ref{thm:Removal}). Using the best known bound for the Removal
lemma,~\cite{Fox:Removal}, would yield $n_{0}$ a tower of height
$\Omega(\log~d)$, 
\[
n_{0}=2^{\left.2^{\cdots^{2}}\right\} \Omega(\log~d)}\;.
\]
Further, the hidden constant in $\Omega(\log~d)$ is quite large,
resulting in a huge $n_{0}$ already for $d=4$.

It is natural to expect that the conclusion of Corollary~\ref{cor:fhgamma}
holds for flag triangulations of any size, not just sufficiently large.
For $(d-1)$-dimensional flag spheres with $d$ even this was conjectured
in~\cite{nevolutz}, and very recently confirmed when $d=4$ and
when $d=6,\,F=f_{1}$ by Zheng~\cite{Zheng:UpperBound35}. Her proof
uses very different methods from ours.

We conjecture that the extremum in Theorem~\ref{thm:Main} is stable,
in the sense that if $F(M)$ is sufficiently close to $F(\mathbf{J}_{\frac{d}{2}}(n))$
then $M$ is still a join of cycles of total length $n$ and of individual
lengths close to $\frac{n}{d/2}$, but not necessarily all equal (see
also \cite[Conjecture 5.1]{AdaHla:DenseFlagTriangulations}).

As mentioned in the introduction, Gal~\cite{SRGal} conjectures that
for flag triangulations of homology spheres $M$ the $\gamma$-vector
$\gamma(M)$ is non-negative. This is known to be true in a number
of special cases (see~\cite{karu2006cd,nevo2011gamma,nevo2011vector}
and the references therein). One method of showing non-negativity
is to exhibit a simplicial complex of which $\gamma(M)$ is the $f$-vector.
In particular, Nevo and Petersen~\cite[Problem 6.4]{nevo2011gamma}
asked if for every $n$-vertex flag triangulation of a homology sphere
$M$ there exists a graph $G$ such that the $\gamma$-vector of $M$
is the clique vector of $G$. Our result $\gamma_{i}(M)\leq\gamma_{i}(\mathbf{J}_{r}(n))$
supports this claim in odd dimension $2r-1$ for large $n$. Indeed,
one checks that $\gamma_{1}(M)=n-4r$ and $\gamma_{i}(\mathbf{J}_{r}(n))=e_{i}(T_{r}(n-4r))$.
If the conjectural graph $G$ exists, then it is $K_{r+1}$-free,
has $n-4r$ vertices, and thus by Zykov's theorem 
\begin{equation}
\gamma_{i}(M)=e_{i}(G)\leq e_{i}(T_{r}(n-4r))=\gamma_{i}(\mathbf{J}_{r}(n)),\label{eq:zzz1}
\end{equation}
which is what we showed. 

Having proved that $F(M)\leq F(J_{r}(n))$ it is tempting to conjecture,
for the classical enumeration vectors (i.e. the $f$-, $h$-, $g$-
or $\gamma$-vector), a generalization in the spirit of Theorem~\ref{thm:Zykov}.
We pose this as an open problem. 
\begin{problem}
\label{prob:bigZykov}Let $(v_{1},\ldots,v_{r})$ be any of $(f_{0},\ldots,f_{r-1})$,
$(h_{1},\ldots,h_{r})$, $(g_{1},\ldots,g_{r})$ or $(\gamma_{1},\ldots,\gamma_{r})$.
Is it true that for sufficiently large $n$ the inequalities 
\begin{equation}
1=\frac{v_{1}(M)}{v_{1}(\mathbf{J}_{r}(n))}\ge\frac{v_{2}(M)}{v_{2}(\mathbf{J}_{r}(n))}\ge\cdots\ge\frac{v_{r}(M)}{v_{r}(\mathbf{J}_{r}(n))}\;\label{eq:bigZykov}
\end{equation}
hold for any flag triangulation of a homology $(2r-1)$-manifold (or
sphere) $M$ with $n$ vertices? 

If $v_{i}=\gamma_{i}$ and $M$ is a homology $(2r-1)$-sphere the
positive answer to Problem~\ref{prob:bigZykov} would follow directly
from the conjecture of Nevo and Petersen mentioned earlier. In this
case the inequalities~(\ref{eq:bigZykov}) are equivalent to those
of Theorem~\ref{thm:Zykov} for some graph $G$ such that $\gamma(M)=e_{2}(G)$.
\end{problem}
It is very likely that $\mathbf{J}_{n}(r)$ is the maximizer of face
numbers for a wider class of $(2r-1)$-dimensional \emph{flag weak
pseudomanifolds}. A weak pseudomanifold of dimension $d-1$ is a pure
$(d-1)$-dimensional simplicial complex in which every face of dimension
$d-2$ belongs to exactly two maximal faces. For $X(G)$ the condition
translates to saying that every maximal clique in $G$ has size $d$
and for every clique $\sigma$ of size $d-1$ the link $\mathrm{lk}_{G}(\sigma)$
consists of two isolated vertices. 
\begin{problem}
Let $r\geq2$. Is it true that:
\begin{enumerate}[label=(\emph{\roman*})]
\item \label{enu:P18-1}for every $n$-vertex flag weak $(2r-1)$-pseudomanifold
$M$ with $n$ sufficiently large we have $f_{i}(M)\leq f_{i}(\mathbf{J}_{r}(n))$
for $i=1,\ldots2r-1$? This is open even for $i=1$. 
\item \label{enu:P18-2}for every $\beta>0$ there is a constant $n_{0}$
such that for every flag weak $(2r-1)$-pseudomanifold $M$ with $n\geq n_{0}$
vertices we have $f_{r}(M)\leq\beta n^{r+1}$ ?
\end{enumerate}
\end{problem}
Recall that for homology manifolds condition~\ref{enu:P18-2}, guaranteed
by Lemma~\ref{lem:middleDS} and ultimately by the middle Dehn--Sommerville
equation, is the weakest possible assumption which allows us to initiate
the stability method for dense graphs. The first author proved in
~\cite{Ada:UpperBound} that for families of flag weak $(2r-1)$-pseudomanifolds
which satisfy a stronger condition $f_{r}(M)\leq Cn^{r}$ for some
fixed $C$ we have $f_{1}(M)\leq f_{1}(\mathbf{J}_{r}(n))$ for sufficiently
large $n$.

In even dimensions the situation seems to be more complicated. For
$r\geq1$ let $J_{r}^{*}(n)$ be the graph obtained from $J_{r}(n-2)$
by adding two new vertices adjacent to all of $J_{r}(n-2)$. Then
the clique complex $\mathbf{J}_{r}^{*}(n):=X(J_{r}^{*}(n))$ is a
flag simplicial $2r$-sphere.
\begin{conjecture}
\label{conj:even}Fix $r\geq1$. For every flag triangulation of a
homology $2r$-sphere $M$ with $n$ vertices we have $f_{i}(M)\le f_{i}(\mathbf{J}_{r}^{*}(n))$
for $i=0,\ldots,2r$ and $\gamma_{i}(M)\leq\gamma_{i}(\mathbf{J}_{r}^{*}(n))$
for $i=0,\ldots,r$.
\end{conjecture}
This statement is obviously true when $r=1$, since $(f_{0}(M),f_{1}(M),f_{2}(M))=(n,2n-4,3n-6)$
and all the inequalities are equalities for every $M$. It is also
true for $r=2$ \cite[Theorem 3.1.3]{SRGal}. As in the odd-dimensional
case, the conjecture is a special case of \cite[Problem 6.4]{nevo2011gamma}.
Indeed, repeating the argument leading to (\ref{eq:zzz1}) would now
yield 
\[
\gamma_{i}(M)=e_{i}(G)\leq e_{i}(T_{r}(n-4r-2))=\gamma_{i}(\mathbf{J}_{r}(n-2))=\gamma_{i}(\mathbf{J}_{r}^{*}(n)).
\]

Conjecture \ref{conj:even} (if true) cannot be augmented by a uniqueness
statement. To see this, consider the subgraph of $J_{r}^{*}(n)$ induced
by $V_{i}\cup\{a,b\}$, where $V_{i}$ is any part of $V(J_{r}(n-2))=V_{1}\sqcup\cdots\sqcup V_{r}$
and $a,b$ are the two additional vertices. It is a flag triangulation
of $S^{2}$ as the suspension of a cycle. Upon replacing this subgraph
by any other flag triangulation of $S^{2}$ with the same number of
vertices one gets a flag simplicial $2r$-sphere with the same face
numbers as $\mathbf{J}_{r}^{*}(n)$.

\section*{Acknowledgments}

We thank Eran Nevo, Günter Ziegler and an anonymous referee for helpful
comments. This project was carried out while JH was visiting the University
of Copenhagen, and he thanks the Department of Mathematical Sciences
for their hospitality. 

The contents of this publication reflects only the authors' views
and not necessarily the views of the European Commission of the European
Union.

\bibliographystyle{plain}
\bibliography{bibliography}

\begin{thebibliography}{10}

\bibitem{Ada:UpperBound}
M.~Adamaszek.
\newblock An upper bound theorem for a class of flag weak pseudomanifolds,
  2013.
\newblock arXiv:1303.5603.

\bibitem{AdaHla:DenseFlagTriangulations}
M.~Adamaszek and J.~Hladk{\'y}.
\newblock Dense flag triangulations of {$3$}-manifolds via extremal graph
  theory.
\newblock {\em Trans. Amer. Math. Soc.}, 367(4):2743--2764, 2015.

\bibitem{aisbett2014frankl}
N.~Aisbett.
\newblock Frankl--{F}{\"u}redi--{K}alai inequalities on the $\gamma$-vectors of
  flag nestohedra.
\newblock {\em Discrete \& Computational Geometry}, 51(2):323--336, 2014.

\bibitem{metricCharney}
R.~Charney.
\newblock Metric geometry: connections with combinatorics.
\newblock {\em Proceedings of FPSCA Conference, DIMACS Series in Discrete
  Mathematics and Theoretical Computer Science}, 24:55--69, 1996.

\bibitem{CharneyDavis}
R.~Charney and M.~Davis.
\newblock The {E}uler characteristic of a non-positively curved, piecewise
  {E}uclidean manifold.
\newblock {\em Pacific J. Math.}, 171:117--137, 1995.

\bibitem{Erd:Stability}
P.~Erd{\H{o}}s.
\newblock On some new inequalities concerning extremal properties of graphs.
\newblock In {\em Theory of {G}raphs ({P}roc. {C}olloq., {T}ihany, 1966)},
  pages 77--81. Academic Press, New York, 1968.

\bibitem{ErdFraRod:RemovalLemma}
P.~Erd{\H{o}}s, P.~Frankl, and V.~R{\"o}dl.
\newblock The asymptotic number of graphs not containing a fixed subgraph and a
  problem for hypergraphs having no exponent.
\newblock {\em Graphs Combin.}, 2(2):113--121, 1986.

\bibitem{flores1933}
A.~Flores.
\newblock {\"U}ber die {E}xistenz $n$-dimensionaler {K}omplexe die nicht im den
  $\mathbb{R}^{2n}$ topologisch einbettar sind.
\newblock {\em Ergeb. Math. Kolloq.}, 5:17--24, 1933.

\bibitem{Fox:Removal}
J.~Fox.
\newblock A new proof of the graph removal lemma.
\newblock {\em Ann. of Math. (2)}, 174(1):561--579, 2011.

\bibitem{andy:clique}
A.~Frohmader.
\newblock Face vectors of flag complexes.
\newblock {\em Israel J. Math.}, 164:153--164, 2008.

\bibitem{SRGal}
\'S.~R. Gal.
\newblock Real root conjecture fails for five- and higher-dimensional spheres.
\newblock {\em Discrete Comput. Geom.}, 34(2):269--284, 2005.

\bibitem{galewski1980classification}
D.~E. Galewski and R.~J Stern.
\newblock Classification of simplicial triangulations of topological manifolds.
\newblock {\em Annals of Mathematics}, 111:1--34, 1980.

\bibitem{goodarzi}
A.~Goodarzi.
\newblock Convex hull of face vectors of colored complexes.
\newblock {\em European Journal of Combinatorics}, 36:247--250, 2014.

\bibitem{Gromov}
M.~Gromov.
\newblock Hyperbolic groups.
\newblock In {\em Essays in Group Theory (S. M. Gersten, M.S.R.I. Publ. 8
  eds.)}, pages 75--264. Springer--Verlag, 1987.

\bibitem{karu2006cd}
K.~Karu.
\newblock The cd-index of fans and posets.
\newblock {\em Compositio Mathematica}, 142(3):701--718, 2006.

\bibitem{nevolutz}
F.~Lutz and E.~Nevo.
\newblock Stellar theory for flag complexes.
\newblock {\em Mathematica Scandinavica}, 118(1):70--82, 2016.

\bibitem{munkres1984}
J.~R. Munkres.
\newblock Topological results in combinatorics.
\newblock {\em Michigan Math. J.}, 31(1):113--128, 1984.

\bibitem{murai2012cd}
S.~Murai and E.~Nevo.
\newblock On the cd-index and $\gamma$-vector of {S}*-shellable {CW}-spheres.
\newblock {\em Mathematische Zeitschrift}, 271(3-4):1309--1319, 2012.

\bibitem{nevo2011gamma}
E.~Nevo and T~K. Petersen.
\newblock On $\gamma$-vectors satisfying the {K}ruskal--{K}atona inequalities.
\newblock {\em Discrete \& Computational Geometry}, 45(3):503--521, 2011.

\bibitem{nevo2011vector}
E.~Nevo, T.~K. Petersen, and B.~E. Tenner.
\newblock The $\gamma$-vector of a barycentric subdivision.
\newblock {\em Journal of Combinatorial Theory, Series A}, 118(4):1364--1380,
  2011.

\bibitem{NovikUBC}
I.~Novik.
\newblock Upper bound theorems for homology manifolds.
\newblock {\em Israel Journal of Mathematics}, 108(1):45--82, 1998.

\bibitem{RusSze:RemovalLemma}
I.~Z. Ruzsa and E.~Szemer{\'e}di.
\newblock Triple systems with no six points carrying three triangles.
\newblock In {\em Combinatorics ({P}roc. {F}ifth {H}ungarian {C}olloq.,
  {K}eszthely, 1976), {V}ol. {II}}, volume~18 of {\em Colloq. Math. Soc.
  J\'anos Bolyai}, pages 939--945. North-Holland, Amsterdam-New York, 1978.

\bibitem{Sim:Stability}
M.~Simonovits.
\newblock A method for solving extremal problems in graph theory, stability
  problems.
\newblock In {\em Theory of {G}raphs ({P}roc. {C}olloq., {T}ihany, 1966)},
  pages 279--319. Academic Press, New York, 1968.

\bibitem{stanley1975upper}
R.~P Stanley.
\newblock The upper bound conjecture and {C}ohen-{M}acaulay rings.
\newblock {\em Stud. Appl. Math}, 54:135--142, 1975.

\bibitem{stanley2004combinatorics}
R.P. Stanley.
\newblock {\em Combinatorics and Commutative Algebra}.
\newblock Progress in Mathematics. Birkh{\"a}user Boston, 2004.

\bibitem{vanKampen1932}
E.R. van Kampen.
\newblock Komplexe in euklidischen r{\"a}umen.
\newblock {\em Abh. Math. Sem. Univ. Hamburg}, 9:72--78, 1932.

\bibitem{wagner:minors}
U.~Wagner.
\newblock Minors in random and expanding hypergraphs.
\newblock {\em Proc. 27th Annual ACM Symposium on Computational Geometry
  (SoCG)}, pages 351--360, 2011.

\bibitem{Zheng:UpperBound35}
H.~Zheng.
\newblock The flag upper bound theorem for 3- and 5-manifolds, 2015.
\newblock arxiv:1512.06958.

\bibitem{Zykov}
A.~A. Zykov.
\newblock On some properties of linear complexes.
\newblock {\em Mat. Sbornik N.S.}, 24(66):163--188, 1949.

\end{thebibliography}

\end{document}